\documentclass[10pt]{article}
\usepackage{indentfirst, latexsym, amsmath,bm,amssymb}
\usepackage{tikz}

\begin{document}

\newtheorem{theorem}{Theorem}[section]
\newtheorem{corollary}[theorem]{Corollary}
\newtheorem{definition}[theorem]{Definition}
\newtheorem{conjecture}[theorem]{Conjecture}
\newtheorem{question}[theorem]{Question}
\newtheorem{lemma}[theorem]{Lemma}
\newtheorem{proposition}[theorem]{Proposition}
\newtheorem{example}[theorem]{Example}
\newtheorem{remark}[theorem]{Remark}
\newenvironment{proof}{\noindent {\bf
Proof.}}{\rule{3mm}{3mm}\par\medskip}
\newcommand{\pp}{{\it p.}}
\newcommand{\de}{\em}

\newcommand{\JEC}{{\it Europ. J. Combinatorics},  }
\newcommand{\JCTB}{{\it J. Combin. Theory Ser. B.}, }
\newcommand{\JCT}{{\it J. Combin. Theory}, }
\newcommand{\JGT}{{\it J. Graph Theory}, }
\newcommand{\ComHung}{{\it Combinatorica}, }
\newcommand{\DM}{{\it Discrete Math.}, }
\newcommand{\ARS}{{\it Ars Combin.}, }
\newcommand{\SIAMDM}{{\it SIAM J. Discrete Math.}, }
\newcommand{\SIAMADM}{{\it SIAM J. Algebraic Discrete Methods}, }
\newcommand{\SIAMC}{{\it SIAM J. Comput.}, }
\newcommand{\ConAMS}{{\it Contemp. Math. AMS}, }
\newcommand{\TransAMS}{{\it Trans. Amer. Math. Soc.}, }
\newcommand{\AnDM}{{\it Ann. Discrete Math.}, }
\newcommand{\NBS}{{\it J. Res. Nat. Bur. Standards} {\rm B}, }
\newcommand{\ConNum}{{\it Congr. Numer.}, }
\newcommand{\CJM}{{\it Canad. J. Math.}, }
\newcommand{\JLMS}{{\it J. London Math. Soc.}, }
\newcommand{\PLMS}{{\it Proc. London Math. Soc.}, }
\newcommand{\PAMS}{{\it Proc. Amer. Math. Soc.}, }
\newcommand{\JCMCC}{{\it J. Combin. Math. Combin. Comput.}, }
\newcommand{\GC}{{\it Graphs Combin.}, }
\title{Properties of the Hyper-Wiener index as a local function\thanks{
 This work is supported by National Natural Science
Foundation of China (Nos.11531001 and 11271256), The Joint Israel-China Program
(No.11561141001), Innovation Program of Shanghai Municipal Education Commission (No.14ZZ016), Specialized Research Fund for the Doctoral Program of Higher Education (No.20130073110075), Simons Foundation (No.245307) and Training program of Lishui (No.2014RC34).  }}
\author{ Ya-Hong  Chen$^{1,2}$,  Hua Wang$^3$ , Xiao-Dong Zhang$^1$\thanks{Corresponding  author ({\it E-mail address:}
xiaodong@sjtu.edu.cn)}
\\
{\small $^1$Department of Mathematics},
{\small Lishui University} \\
{\small  Lishui, Zhejiang 323000, PR China}\\
{\small $^2$Department of Mathematics, and MOE-LSC,}
{\small Shanghai Jiao Tong University} \\
{\small  800 Dongchuan road, Shanghai, 200240,  P.R. China}\\
{\small $^3$Department of Mathematical Sciences,}
{\small Georgia Southern University }\\
{\small Statesboro, GA 30460 USA}}

\maketitle
 \begin{abstract}
 Hyper-Wiener index was introduced as one of the main generalizations of the well known Wiener index. Through the years properties of the Wiener index have been extensively studied in both Mathematics and Chemistry. The Hyper-Wiener index, although received much attention, is far from being thoroughly examined due to its complex definition. We consider the local version of the Hyper-Wiener index ($WW(G)$), defined as $ww_G(v)=\sum\limits_{u\in V(G)}(d^2(u,v)+d(u,v))$ for a vertex $v$ in a graph $G$, in trees. For established results on the Wiener index ($W(.)$), we present analogous studies on $WW(.)$. In addition to interesting observations, some conjectures and questions are also proposed.
   \end{abstract}


\vskip 0.5cm

\section{Introduction}

The so called topological indices are popular descriptors of structural information that have been vigorously studied. One of the most well known such indices is the Wiener index, defined as
\begin{eqnarray}
W(G)=\frac{1}{2}\sum\limits_{v\in V(G)}\sum\limits_{u\in V(G)}d(u,v)= \frac{1}{2}\sum\limits_{v\in V(G)}w_G(v).
\end{eqnarray}
Here $w_G(v)=\sum\limits_{u\in V(G)}d(u,v)$ is generally considered as the distance function of a vertex, serving as the local version of the $W(.)$ function. Although this concept is generally known as named after the chemist Harry Wiener \cite{wiener}, the study of distance in graphs has long been of interest from pure mathematical point of view, see survey \cite{dobrynin2001,xu2014}.

Since many of the applications of topological indices and in particular the Wiener index deal with acyclic structures, the properties of the Wiener index and the local distance function of trees have been extensively studied \cite{1974adam,Barefoot1997, dobrynin2001,Jordan1869,zelinka1968}. Barefoot et al. \cite{Barefoot1997} determined
extremal values of $w_T(w)/w_T(u)$, $w_T(w)/w_T (v)$, $W(T )/w_T (v)$, and $W(T)/w_T(w)$, where $T$ is a
tree on $n$ vertices, $v$ is in the centroid of the tree $T$, and $u,w$ are leaves in $T$.
Recently, analogous questions have been considered for the number of subtrees \cite{szekely2013,szekely2014} and distance between leaves \cite{wang2014}. The resulted extremal trees are very similar to those of the Wiener index. In addition, the ``middle part" of a tree such as center \cite{Jordan1869}, centroid \cite{Jordan1869,zelinka1968}, leaf-centroid \cite{wang2015} of the tree has historically been of interest from both practical and theoretical points of view.

The Hyper-Wiener index was introduced as one of the most important generalizations of the Wiener index \cite{Randic1993}, defined as
\begin{eqnarray}
WW(G)=\sum\limits_{\{u,v\}\subseteq V(G)}\binom {d(u,v)+1}{2}= \frac{1}{4}\sum\limits_{v\in V(G)}ww_G(v),
\end{eqnarray}
where $ww_G(v)=\sum\limits_{u\in V(G)}(d^2(u,v)+d(u,v))$. Let $S_G(v)=\sum\limits_{u\in V(G)}d^2(u,v)$, then $ww_G(v)=w_G(v)+S_G(v)$. Due to the rather complex definition, the properties of $WW(.)$ and $ww_G(.)$ are far from being sufficiently studied.

Among the limited analogous results between $W(.)$ and $WW(.)$ \cite{Behtoei2011,gutman2002,zhou2004}, it is observed that these two functions behave in rather similar ways. Motivated by $W(.)$ and $w(.)$, in this paper we focus on $ww_T(.)$ of trees and present studies analogous to those of $w_T(.)$. It is not surprising to see that the analysis of $ww_T(.)$ is much more complicated than $w_T(.)$ in most cases. In addition to some interesting observations, we also propose some conjectures and questions.

\section{``Middle part'' of a tree}

It seems that the examination of the ``middle part'' of a tree first started from \cite{Jordan1869}, where $W_T(.)$ was introduced in an equivalent form through {\it branch weight}. A maximal subtree containing a vertex $v$ of a tree $T$ as an end vertex is called a branch of $T$ at $v$. The {\it weight} of a branch $B$, denoted by $bw(B)$, is the number of edges in it. The {\it centroid} of a tree $T$, denoted by $C(T)$, is the set of vertices $v$ of  $T$ for which the maximum branch weight at $v$ is minimized.

Jordan \cite{Jordan1869} has characterized the properties of the centroid of a tree as follows.

\begin{theorem}\label{th3}\cite{Jordan1869}
If $C=C(T)$ is the centroid of a tree $T$ of order $n$ then one of the following holds:\\
(i) $C=:\{c\}$ and $bw(c)\leq (n-1)/2$,\\
(ii) $C=:\{c_1,c_2\}$ and $bw(c_1)=bw(c_2)=n/2$.\\
In both cases, if $v\in V(T)\backslash C$, then $bw(v)>n/2$.
\end{theorem}

\begin{remark}
In the context of $w_T(.)$, the centroid $C(T)$ of a tree is the set of vertices with minimum $w_T(.)$ value and it is well known that $C(T)$ contains one or two adjacent vertices.
\end{remark}

With respect to $ww_T(.)$, we define the {\it hyper-centroid} of a tree $T$, denoted by $C_w(T)$, as the set of vertices in $T$ minimizing $ww_T(.)$. The hyper-centroid is the natural analogue of the well-known concepts of center and centroid of a tree. First we show the following observation on $ww_T(.)$.

\begin{proposition}\label{lemma1}
For any three vertices $x,y,z\in V(T)$ such that $xz,yz\in E(T)$, we must have
$$2ww_T(z)<ww_T(x)+ww_T(y).$$
\end{proposition}
\begin{proof}
Consider the connected components $T_x$, $T_y$ and $T_z$ in $T-\{xz,yz\}$ that contain $x$, $y$ and $z$, respectively. See Figure~\ref{fig:ex_xyz}.

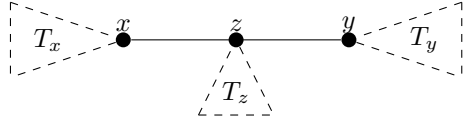
\begin{figure}[htbp]
\centering
    \begin{tikzpicture}[scale=1]
        \node[fill=black,circle,inner sep=2pt] (t1) at (0,0) {};
        \node[fill=black,circle,inner sep=2pt] (t2) at (1.5,0) {};
        \node[fill=black,circle,inner sep=2pt] (t8) at (3,0) {};

        \draw (t1)--(t8);
        \draw [dashed] (-1.5,.5)--(t1);
        \draw [dashed] (t1)--(-1.5,-.5);
        \draw [dashed] (-1.5,.5)--(-1.5,-.5);
        \draw [dashed] (4.5,.5)--(t8);
                \draw [dashed] (t8)--(4.5,-.5);
                \draw [dashed] (4.5,.5)--(4.5,-.5);
        \draw [dashed] (1,-1)--(t2);
                \draw [dashed] (t2)--(2,-1);
                \draw [dashed] (1,-1)--(2,-1);

        \node at (0,.2) {$x$};
        \node at (3,.2) {$y$};
        \node at (1.5,.2) {$z$};

        \node at (-1,0) {$T_x$};
                \node at (4,0) {$T_y$};
\node at (1.5,-.7) {$T_z$};


        \end{tikzpicture}
\caption{The vertices $x$, $y$, $z$ and the subtrees $T_x$, $T_y$, $T_z$.}\label{fig:ex_xyz}
\end{figure}

Through examining the distance from a vertex to the vetices in these components we have

$$ww_T(x)=\sum\limits_{p\in V(T_x)}(d^2(x,p)+d(x,p))+\sum\limits_{p\in V(T_z)}(d^2(z,p)+3d(z,p)+2)+\sum\limits_{p\in V(T_y)}(d^2(y,p)+5d(y,p)+6), $$

$$ww_T(y)=\sum\limits_{p\in V(T_y)}(d^2(y,p)+d(y,p))+\sum\limits_{p\in V(T_z)}(d^2(z,p)+3d(z,p)+2)+\sum\limits_{p\in V(T_x)}(d^2(x,p)+5d(x,p)+6)$$
and
$$ww_T(z)=\sum\limits_{p\in V(T_z)}(d^2(z,p)+d(z,p))+\sum\limits_{p\in V(T_x)}(d^2(x,p)+3d(x,p)+2)+\sum\limits_{p\in V(T_y)}(d^2(y,p)+3d(y,p)+2).$$

Direct calculations then yield
$$ww_T(x)+ww_T(y)-2ww_T(z)=2\left(|V(T_x)|+|V(T_y)|+2\sum\limits_{p\in V(T_z)}(d(z,p)+1)\right)>0. $$
\end{proof}

Proposition \ref{lemma1} implies that $ww(.)$ is strictly convex along any path of $T$. As an immediate consequence, we have

\begin{corollary}\label{cor:1}
Given a tree $T$:
\begin{enumerate}
\item[(i)] on any path of $T$, there are at most two adjacent vertices with the smallest $ww_T(.)$.
\item[(ii)] on any maximum path of $T$, the largest $ww_T(.)$ is obtained at a leaf.
\end{enumerate}
\end{corollary}

In return, Corollary~\ref{cor:1} yields the following statement for $C_w(T)$ analogous to that for $C(T)$.

\begin{theorem}\label{lemma2}
The subgraph induced by $C_w(T)$ is either a single vertex or two vertices joined by an edge.
\end{theorem}

\begin{proof}
For any two vertices in $C_w(T)$, say $u$ and $v$, they are on a common path in $T$ and hence must be adjacent to each other by part (i) of Corollary~\ref{cor:1}. Thus $C_w(T)$ induces a complete subgraph, which is not possible in a tree $T$ if $|C_w(T)|\geq 3$.
\end{proof}

As an important part of studies on $w_T(.)$ and $C(T)$, the following fact analogous to Theorem~\ref{th3} has been frequently used.

\begin{proposition}
For a vertex $v \in C(T)$ and a vertex $u$ adjacent to $v$, we must have
\begin{equation}\label{eq:1}
n_{vu}(v) \geq n_{vu}(u),
\end{equation}
where $n_{vu}(v)$ ($n_{vu}(u)$) denotes the number of vertices closer to $v$ ($u$) than $u$ ($v$) in $T$, with equality if and only if $u\in C(T)$.
\end{proposition}

When the similar property is considered for $ww_T(.)$ and $C_w(T)$, we have the following.
\begin{theorem}\label{lemma3}
Let $T$ be a tree of order $n$. For two adjacent vertices $u,v\in V(T)$, we have
$$ww_T(v)-ww_T(u)=2(w_{T_u}(u)-w_{T_v}(v)+|V(T_u)|-|V(T_v)|).$$
 Moreover, if $v \in C_w(T)$, we must have
$$n_{vu}(v) + w_{T_v}(v) \geq n_{vu}(u) + w_{T_u}(u),$$
where $T_v$ ($T_u$) is the connected component containing $v$ ($u$) in $T-uv$, with equality if and only if $u\in C_w(T)$.
\end{theorem}

\begin{proof}
From the definition we have
$$ww_T(v)=\sum\limits_{w\in V(T_v)}(d^2(v,w)+d(v,w))+\sum\limits_{w\in V(T_u)}(d^2(v,w)+d(v,w))$$
and
$$ww_T(u)=\sum\limits_{w\in V(T_u)}(d^2(u,w)+d(u,w))+\sum\limits_{w\in V(T_v)}(d^2(u,w)+d(u,w)).$$
Then
\begin{align*}
ww_T(v)-ww_T(u) & = \sum\limits_{w\in V(T_v)}(d(v,w)-d(u,w))(d(v,w)+d(u,w)+1)\\
& \quad\quad\quad\quad +\sum\limits_{w\in V(T_u)}(d(v,w)-d(u,w))(d(v,w)+d(u,w)+1)\\
& = -\sum\limits_{w\in V(T_v)}(d(v,w)+d(u,w)+1)+\sum\limits_{w\in V(T_u)}(d(v,w)+d(u,w)+1) \\
& = -\sum\limits_{w\in V(T_v)}(2d(v,w)+2)+\sum\limits_{w\in V(T_u)}(2d(u,w)+2)\\
& =2(w_{T_u}(u)-w_{T_v}(v)+|V(T_u)|-|V(T_v)|) \\
& =2(w_{T_u}(u)-w_{T_v}(v)+n_{vu}(u)-n_{vu}(v)).
\end{align*}
Given that $v \in C_w(T)$ and hence $ww_T(v) \leq ww_T(u)$, we have
$$n_{vu}(v) + w_{T_v}(v) \geq n_{vu}(u) + w_{T_u}(u)$$
with equality if and only if $u\in C_w(T)$.
\end{proof}

\section{Distance between $C(T)$ and $C_w(T)$}
Given Theorem~\ref{lemma2} for $C_w(T)$ and the similar statement for $C(T)$, it is interesting to explore their relations. It is easy to find trees (such as the star and the path) where $C_w(T) = C(T)$. But in general these two middle parts are not the same. A natural question is how far apart can $C_w(T)$ and $C(T)$ be in a tree on $n$ vertices. We consider exactly this question in this section.

\begin{definition}
An $r$-comet or order $n$, denoted by $T(n,r)$, is a tree resulted from identifying the center of a star on $n-r+1$ verties with one end vertex of a path on $r$ vertices (Figure~\ref{fig:com}).
\end{definition}

\begin{figure}[htbp]
\centering
    \begin{tikzpicture}[scale=1]
        \node[fill=black,circle,inner sep=1pt] (t1) at (0,0) {};
        \node[fill=black,circle,inner sep=1pt] (t2) at (1,0) {};
        \node[fill=black,circle,inner sep=1pt] (t3) at (2,0) {};
        \node[fill=black,circle,inner sep=1pt] (t8) at (3,0) {};

        \node[fill=black,circle,inner sep=1pt] (t5) at (-.5,-.5) {};
        \node[fill=black,circle,inner sep=1pt] (t6) at (-.5,0) {};
        \node[fill=black,circle,inner sep=1pt] (t7) at (-.5,.5) {};

        \draw (t1)--(t2);
        \draw (t3)--(t8);
        \draw (t5)--(t1);
        \draw (t6)--(t1);
        \draw (t7)--(t1);


        \node at (1.5,0) {$\ldots$};
        \node at (1.5,-.5) {$\underbrace{\hspace{7.8 em}}_{r \hbox{ vertices}}$};

        \end{tikzpicture}
\caption{An $r$-Comet $T(n,r)$}\label{fig:com}
\end{figure}
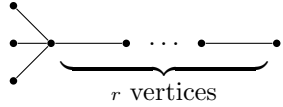

\begin{theorem}
For a tree $T$ on $n$ vertices with $v\in C_w(T)$ and $u\in C(T)$,
$$ \min d(v,u) \leq \left\lfloor\frac{n-1}{8}\right\rfloor $$
when $n$ is odd, with equality achieved when $T$ is a $\frac{n+1}{2}$-comet; and
$$  \min d(v,u) \leq \left\lfloor\frac{n^2-2n-8}{8n+8}\right\rfloor $$
when $n$ is even, with equality achieved when $T$ is a $\frac{n}{2}$-comet.
\end{theorem}

\begin{proof}
Let $d(v,u)=x$ be the minimum distance between vertices in $C_w(T)$ and $C(T)$, and denote by $T_v$ ($T_u$) the connected component containing $v$ ($u$) in the graph resulted from removing edges on the path connecting $v$ and $u$ (Figure~\ref{fig:ex}).

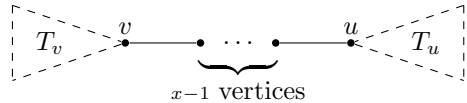
\begin{figure}[htbp]
\centering
    \begin{tikzpicture}[scale=1]
        \node[fill=black,circle,inner sep=1pt] (t1) at (0,0) {};
        \node[fill=black,circle,inner sep=1pt] (t2) at (1,0) {};
        \node[fill=black,circle,inner sep=1pt] (t3) at (2,0) {};
        \node[fill=black,circle,inner sep=1pt] (t8) at (3,0) {};

        \draw (t1)--(t2);
        \draw (t3)--(t8);
        \draw [dashed] (-1.5,.5)--(t1);
        \draw [dashed] (t1)--(-1.5,-.5);
        \draw [dashed] (-1.5,.5)--(-1.5,-.5);
        \draw [dashed] (4.5,.5)--(t8);
                \draw [dashed] (t8)--(4.5,-.5);
                \draw [dashed] (4.5,.5)--(4.5,-.5);

        \node at (0,.2) {$v$};
        \node at (3,.2) {$u$};

        \node at (-1,0) {$T_v$};
                \node at (4,0) {$T_u$};

        \node at (1.5,0) {$\ldots$};
        \node at (1.5,-.5) {$\underbrace{\hspace{3 em}}_{x-1 \hbox{ vertices}}$};

        \end{tikzpicture}
\caption{The vertices $v$, $u$ and the subtrees $T_v$, $T_u$}\label{fig:ex}
\end{figure}

First consider the case of odd $n$, \eqref{eq:1} implies that
\begin{equation}\label{eq:2}
|V(T_u)| \geq |V(T-T_u)|+1
\end{equation}
and Theorem~\ref{lemma3} implies that
\begin{equation}\label{eq:3}
|V(T_v)| + w_{T_v}(v) > |V(T-T_v)| + w_{T-T_v}(w),
\end{equation}
where $w$ is the neighbor of $v$ on the path connecting $v$ and $u$.

Let $y= |V(T_v)|$, we have
$$ {y \choose 2} \geq w_{T_v}(v) $$
with equality if and only if $T_v$ is a path. On the other hand, we have
$$ |V(T-T_v)| \geq |V(T_u)| + x-1 \geq |V(T-T_u)|+1 +(x - 1) \geq y + (x-1) + 1 + (x-1) = 2x+y-1  $$
from \eqref{eq:2}, with equality if and only if every vertex on the path connecting $v$ and $u$ is of degree 2 and $|V(T_u)|=|V(T-T_u)|+1=x+y$; and
$$ w_{T-T_v}(w) \geq {x \choose 2} + x(|V(T_u)| -1) \geq {x \choose 2} + x(x+y -1) $$
with equality if and only if every vertex on the path connecting $v$ and $u$ is of degree 2, $T_u$ is a star, and $|V(T_u)|=|V(T-T_u)|+1=x+y$.
Hence from inequality~\eqref{eq:3} we have
\begin{equation}\label{eq:4}
y + {y \choose 2} \geq |V(T_v)| + w_{T_v}(v) \geq |V(T-T_v)| + w_{T-T_v}(w) + 1 \geq (2x+y-1) + {x \choose 2} + x(x+y -1) + 1
\end{equation}
with possible equality only if $T_u$ is a star, $T-T_u$ is a path, and $|V(T_u)|=|V(T-T_u)|+1=x+y$. Letting $x+y=z$, \eqref{eq:4} is equivalent to
$$ x \leq \frac{z^2 - z}{4z} =: f(z) . $$
It is easy to check $f'(z)=\frac{1}{4} > 0$ and hence $f(z)$ achieves its maximum when $z=x+y=|V(T_u)|=\frac{n+1}{2}$. Hence
 $$\max x =\left\lfloor\frac{n-1}{8}\right\rfloor$$ when $n$ is odd and $T$ is a $\frac{n+1}{2}$-comet.

Following the same argument, if $n$ is even, we have
$$
|V(T_u)| \geq |V(T-T_u)|+2
$$
and
$$ y + {y \choose 2} \geq (2x+y) + {x \choose 2} + x(x+y) +1 $$
with possible equality only if $T_u$ is a star, $T-T_u$ is a path, and $|V(T_u)|=|V(T-T_u)|+2=x+y+1$. This simplifies to
$$ x \leq \frac{z^2 -z-2}{4z+2} $$
maximized when $z=x+y=|V(T_u)|-1=\frac{n}{2}$. Hence
 $$\max x =\left\lfloor\frac{n^2-2n-8}{8n+8}\right\rfloor$$ when $n$ is even and $T$ is a $\frac{n}{2}$-comet.
\end{proof}

\section{Extremal ratios}
As it was established that the minimum $w_T(.)$ is obtained at the centroid vertices and the maximum $w_T(.)$ is obtained at a leaf, the extremal values of $\frac{w_T(w)}{w_T(u)}$ and $\frac{w_T(w)}{w_T(v)}$ (where $v \in C(T)$, $u$ and $w$ are leaves) have been studied in \cite{Barefoot1997} along with other extremal ratios. Similarly, we have already seen that the minimum $ww_T(.)$ is obtained at the hyper-centroid vertices and the maximum $ww_T(.)$ is obtained at a leaf. In this section we explore the extremal values of the analogous ratios
$\frac{ww_T(w)}{ww_T(u)}$ and $\frac{ww_T(w)}{ww_T(v)}$ where $v \in C_w(T)$, $u$ and $w$ are leaves. These questions turned out to be rather complicated and we propose some questions.

\subsection{Extremal values of $ww_T(w)/ww_T(u)$ where $u$ and $w$ are leaves}

We start with the following simple observation.

\begin{proposition}\label{prop:1}
Let $T$ be a tree with leaves $u$ and $w$ such that the maximum $\frac{ww_T(w)}{ww_T(u)}$ is achieved. Then all internal vertices on the path connecting $u$ and $w$, except possibly for the neighbor of $u$, must be of degree 2.
\end{proposition}

\begin{proof}
Let $T$ be a tree with leaves $w$ and $u$ such that $\frac{ww_T(w)}{ww_T(u)}$ is maximized among all trees on $n$ vertices.
Let $u=u_0 u_1 \cdots u_r=w$ be the $u$-$w$ path in $T$ and note that $2\leq r\leq n-1$.

Suppose, for contradiction, that for some $2\leq i\leq r-1$, the vertex $u_i$ has a neighbor $x$ different from $u_{i-1}$ and $u_{i+1}$. Let $T'$ be the tree obtained from $T$ by deleting $xu_i$ and adding $xu_{i-1}$, i.e., $T'=T-xu_i+xu_{i-1}$. It is easy to see that $ww_{T'}(w)>ww_T(w)$ and $ww_{T'}(u)<ww_T(u)$, then
$$\frac{ww_{T'}(w)}{ww_{T'}(u)}>\frac{ww_{T}(w)}{ww_T(u)},$$
a contradiction.
\end{proof}

With Proposition~\ref{prop:1}, let $T_B$ denote the connected component containing $u_1$ after removing all edges on the $u$-$w$ path. See Figure~\ref{fig:ex_tb}.

\begin{figure}[htbp]
\centering
    \begin{tikzpicture}[scale=1.3]
        \node[fill=black,circle,inner sep=1.5pt] (t1) at (0,0) {};
        \node[fill=black,circle,inner sep=1.5pt] (t2) at (1,0) {};
        \node[fill=black,circle,inner sep=1.5pt] (t3) at (2,0) {};
        \node[fill=black,circle,inner sep=1.5pt] (t8) at (3,0) {};

        \draw (t1)--(t2);
        \draw (t3)--(t8);
        \draw [dashed] (.5,-1)--(t2);
        \draw [dashed] (t2)--(1.5,-1);
        \draw [dashed] (.5,-1)--(1.5,-1);

\node at (0,.2) {$u$};
        \node at (1,.2) {$u_1$};
        \node at (2,.2) {$u_{r-1}$};
        \node at (3,.2) {$w$};

        \node at (1,-.7) {$T_B$};

        \node at (1.5,0) {$\ldots$};

        \end{tikzpicture}
\caption{The vertices $u$, $w$ and the subtrees $T_B$.}\label{fig:ex_tb}
\end{figure}
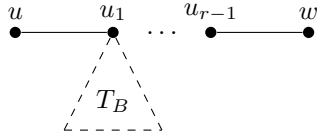

Then applying Theorem~\ref{lemma3} (repeatedly) yields
$$ ww_T(u) = ww_T(u_1)+2[W_{T_B}(u_1)+r(r-1)/2+n-2] $$
and

\begin{align*}
& ww_T(w) \\
= & ww_T(u_1)+2(r-1)w_{T_B}(u_1)+2\sum\limits_{j=2}^{r-1}d(u_1,u_j)|V(T_B)|+2\sum\limits_{k=2}^{r-2}\sum\limits_{j=k+1}^{r-1}d(u_k,u_j)\\
&\quad\quad\quad\quad\quad\quad +2\sum\limits_{j=1}^{r-1}d(u_0,u_j) -2\sum\limits_{k=2}^r\sum\limits_{j>k}^rd(u_k,u_j)+2(r-1)(n-r) \\
= & ww_T(u_1)+2[(r-1)w_{T_B}(u_1)+(r-1)(nr-r^2+2)/2] .
\end{align*}

Hence
\begin{equation*}
\begin{aligned}
\frac{ww_T(w)}{ww_T(u)}&=\frac{ww_T(u_1)+2[(r-1)w_{T_B}(u_1)+(r-1)(nr-r^2+2)/2]}{ww_T(u_1)+2[w_{T_B}(u_1)+r(r-1)/2+n-2]}\\
&=1+3\frac{2(r-2)w_{T_B}(u_1)+(r-1)(nr-r^2-r+2)-2(n-2)}{9w_{T_B}(u_1)+3S_{T_B}(u_1)+6(n-1)+(r+4)r(r-1)} .
\end{aligned}
\end {equation*}

The above formula allows quick computation of the ratio based on the information of $T_B$ alone. Computation results based on this formula suggests the following, which we post as a question.

\begin{question}\label{th1}
For leaves $w$ and $u$ in a tree $T$ of order $n\geq 8$, let the integers $k\geq1$ and $s$ be defined by $4n=k^2+s$, $0\leq s\leq2k$. Is it true that
$$\frac{ww_T(w)}{ww_T(u)}\leq 1+3\frac{(r-1)[-r^2+(n-3)r+2n]+2r-4n+6}{(r+4)r(r-1)+6(3n-2r-3)},$$
where $$r=\left\{\begin{array}{ll}
 \lfloor 2\sqrt{n}\rfloor-2, &  0\leq s\leq k-6,\\
 \lfloor 2\sqrt{n}\rfloor-1 ,  & k-5\leq s\leq 2k,
\end{array}\right.$$
with equality when $T=T(n,r)$?
\end{question}

As examples, Table~\ref{table1} shows the structures of extremal trees for small $n$. Note that when $n=7$, the extremal tree $T_1$ is obtained from $P_6$ by joining one of its middle vertices to an additional leaf. It is interesting to see that in this case (unlike all other extremal structures) the $u$-$w$ path does not form the diameter of the tree.

\setlength{\unitlength}{.1in}
\begin{table}\begin{center}
\begin{tabular}{lccc}
\hline

n & \mbox{Graph} & \mbox{Value} & \mbox{Structure of extremal trees with order n}\\ \hline
2 & $P_2$ & 1 &
\begin{picture}(15,2)
\linethickness{0.25mm}
\put(0.2,0.4){\line(1,0){3}}
\multiput(0,0)(3,0){2}{$\bullet$}
\end{picture}
\\ \hline
3 & $P_3$ & 1 &
\begin{picture}(15,2)
\linethickness{0.25mm}
\put(0.2,0.4){\line(1,0){6}}
\multiput(0,0)(3,0){3}{$\bullet$}
\end{picture}
\\ \hline
4 & $P_4,S_4$ &1&
\begin{picture}(18,3)
\linethickness{0.25mm}
\put(0.2,0.4){\line(1,0){9}}
\put(10,0.4){\line(1,0){6}}
\put(13,0.4){\line(0,1){2}}
\multiput(0,0)(3,0){4}{$\bullet$}
\multiput(9.7,0)(3,0){2}{$\bullet$}
\multiput(12.7,2)(0,0){1}{$\bullet$}
\multiput(15.7,0)(0,0){1}{$\bullet$}
\end{picture}
\\ \hline
5 &$T(5,3)$& 16/13&
\begin{picture}(15,3)
\linethickness{0.25mm}
\put(0.2,0.4){\line(1,0){9}}
\put(3.3,0.4){\line(0,1){2}}
\multiput(0,0)(3,0){4}{$\bullet$}
\multiput(3,2)(0,0){1}{$\bullet$}
\end{picture}
\\ \hline
6 &$T(6,3)$& 11/8&
\begin{picture}(15,3)
\linethickness{0.25mm}
\put(0.2,0.4){\line(1,0){9}}
\put(3.3,0.4){\line(0,1){2}}
\put(3.3,0.4){\line(-1,1){1.5}}
\multiput(0,0)(3,0){4}{$\bullet$}
\multiput(3,2)(0,0){1}{$\bullet$}
\multiput(1.4,1.7)(0,0){1}{$\bullet$}
\end{picture}
\\ \hline
7 &$T_1$& 45/29&
\begin{picture}(15,3)
\linethickness{0.25mm}
\put(7,2){$u$}
\put(15,1){$w$}
\put(0.2,0.4){\line(1,0){15}}
\put(6.3,0.4){\line(0,1){2}}
\multiput(0,0)(3,0){6}{$\bullet$}
\multiput(6,2)(0,0){1}{$\bullet$}
\end{picture}
\\ \hline
8 &$T(8,4)$& 50/29&
\begin{picture}(15,3)
\linethickness{0.25mm}
\put(0.2,0.4){\line(1,0){12}}
\put(3.3,0.4){\line(0,1){2}}
\put(3.3,0.4){\line(-1,1){1.5}}
\put(3.3,0.4){\line(1,1){1.5}}
\multiput(0,0)(3,0){5}{$\bullet$}
\multiput(3,2)(0,0){1}{$\bullet$}
\multiput(1.4,1.7)(0,0){1}{$\bullet$}
\multiput(4.4,1.6)(0,0){1}{$\bullet$}
\end{picture}
\\ \hline
9 &$T(9,4)$& 15/8&
\begin{picture}(15,3)
\linethickness{0.25mm}
\put(0.2,0.4){\line(1,0){12}}
\put(3.3,0.4){\line(0,1){2}}
\put(3.3,0.4){\line(-1,1){1.5}}
\put(3.3,0.4){\line(1,1){1.5}}
\put(3.3,0.4){\line(-2,1){1.9}}
\multiput(0,0)(3,0){5}{$\bullet$}
\multiput(3,2)(0,0){1}{$\bullet$}
\multiput(1.4,1.7)(0,0){1}{$\bullet$}
\multiput(1,1.1)(0,0){1}{$\bullet$}
\multiput(4.4,1.6)(0,0){1}{$\bullet$}
\end{picture}
\\ \hline
10 &$T(10,5)$& 95/47&
\begin{picture}(15,3)
\linethickness{0.25mm}
\put(0.2,0.4){\line(1,0){15}}
\put(3.3,0.4){\line(0,1){2}}
\put(3.3,0.4){\line(-1,1){1.5}}
\put(3.3,0.4){\line(1,1){1.5}}
\put(3.3,0.4){\line(-2,1){1.9}}
\multiput(0,0)(3,0){6}{$\bullet$}
\multiput(3,2)(0,0){1}{$\bullet$}
\multiput(1.4,1.7)(0,0){1}{$\bullet$}
\multiput(1,1.1)(0,0){1}{$\bullet$}
\multiput(4.4,1.6)(0,0){1}{$\bullet$}
\end{picture}
\\ \hline
\end{tabular}\\
\end{center}
\caption{The extremal trees of order $n$ with maximum value of $\frac{ww_T(w)}{ww_T(u)}$}\label{table1}
\end{table}
\vspace{0.2cm}

\subsection{Extremal values of $ww_T(w)/ww_T(v)$ where $w$ is leaf and $v\in C_w(T)$}

Similar to the previous section, we have

\begin{proposition}\label{prop:2}
Let $T$ be a tree with $v\in C_w(T)$ and leaf $w$ such that the maximum $\frac{ww_T(w)}{ww_T(v)}$ is achieved. Then all internal vertices on the path connecting $u$ and $w$ must be of degree 2.
\end{proposition}

\begin{proof}
Assume that $T$ is a tree on $n$ vertices with leaf $w$ and $v\in C_w(T)$ such that $\frac{ww_T(w)}{ww_T(v)}$ is maximized. Let $w=w_0w_1\cdots w_{r-1}=v$ be the $w$-$v$ path in $T$.

Suppose, for contradiction, that $r\geq 2$ and for some $1\leq i\leq r-2$, $w_i$ has a neighbor $u$ different from $w_{i-1}$ and $w_{i+1}$. Let $T'=T-uw_i+uw_{i+1}$, it is easy to see that $ww_{T'}(w)>ww_T(w)$ and $ww_{T'}(v)<ww_T(v)$. Hence
$$ \frac{ww_{T'}(w)}{ww_{T'}(v')} \geq \frac{ww_{T'}(w)}{ww_{T'}(v)}>\frac{ww_T(w)}{ww_T(v)}$$
where $v'$ (possibly equal to $v$) is in $C_w(T')$, a contradiction.
\end{proof}

Now let $T_B$ be the connected component containing $v$ after removing all edges on the $v$-$w$ path. Similar computation as the previous section yields
\begin{equation*}
\frac{ww_T(w)}{ww_T(v)}=1+\frac{2(r-1)w_{T_B}(v)+(r-1)r(n-r)}{\frac{r(r^2-1)}{3}+w_{T_B}(v)+S_{T_B}(v)} .
\end{equation*}

Once again our computation leads to the following:

\begin{question}\label{th2}
If $T$ is a tree of order $n\geq2$ with leaf $w$ and $v\in C_w(T)$, is it true that
\begin{equation}
\frac{ww_T(w)}{ww_T(v)}\leq\frac{-2r^3+3nr^2-3r^2+3nr-r}{r^3-7r+6n},
\end{equation}
where
 $$r=\left\{\begin{array}{ll}
 \lfloor \sqrt{2n}\rfloor-1, &  0\leq s\leq k-4,\\
 \lfloor \sqrt{2n}\rfloor ,  & k-3\leq s\leq 2k,
\end{array}\right.$$
with equality when $T=T(n,r)$?
\end{question}

Table~\ref{table2} presents such extremal trees for small $n$. Note that such extremal structures are not necessarily unique, as can be seen from the case of $n=9$.

\setlength{\unitlength}{.1in}
\begin{table}
\begin{center}
\begin{tabular}{lccc}\hline
n & \mbox{Graph} & \mbox{Value} & \mbox{Structure of extremal trees with order n}\\ \hline
2 & $P_2$ & 1 &
\begin{picture}(15,2)
\linethickness{0.25mm}
\put(0.2,0.4){\line(1,0){3}}
\multiput(0,0)(3,0){2}{$\bullet$}
\end{picture}
\\ \hline
3 & $P_3$ & 2 &
\begin{picture}(15,2)
\linethickness{0.25mm}
\put(0.2,0.4){\line(1,0){6}}
\multiput(0,0)(3,0){3}{$\bullet$}
\end{picture}
\\ \hline
4 & $S_4$ &7/3&
\begin{picture}(15,3)
\linethickness{0.25mm}
\put(0.2,0.4){\line(1,0){6}}
\put(3.3,0.4){\line(0,1){2}}
\multiput(0,0)(3,0){3}{$\bullet$}
\multiput(3,2)(0,0){1}{$\bullet$}
\end{picture}
\\ \hline
5 &$T(5,3)$& 8/3&
\begin{picture}(15,3)
\linethickness{0.25mm}
\put(0.2,0.4){\line(1,0){9}}
\put(3.3,0.4){\line(0,1){2}}
\multiput(0,0)(3,0){4}{$\bullet$}
\multiput(3,2)(0,0){1}{$\bullet$}
\end{picture}
\\ \hline
6 &$T(6,3)$& 22/7&
\begin{picture}(15,3)
\linethickness{0.25mm}
\put(0.2,0.4){\line(1,0){9}}
\put(3.3,0.4){\line(0,1){2}}
\put(3.3,0.4){\line(-1,1){1.5}}
\multiput(0,0)(3,0){4}{$\bullet$}
\multiput(3,2)(0,0){1}{$\bullet$}
\multiput(1.4,1.7)(0,0){1}{$\bullet$}
\end{picture}
\\ \hline
7 &$T(7,3)$&7/2&
\begin{picture}(15,3)
\linethickness{0.25mm}
\put(0.2,0.4){\line(1,0){9}}
\put(3.3,0.4){\line(0,1){2}}
\put(3.3,0.4){\line(-1,1){1.5}}
\put(3.3,0.4){\line(1,1){1.5}}
\multiput(0,0)(3,0){4}{$\bullet$}
\multiput(3,2)(0,0){1}{$\bullet$}
\multiput(1.4,1.7)(0,0){1}{$\bullet$}
\multiput(4.4,1.6)(0,0){1}{$\bullet$}
\end{picture}
\\ \hline
8 &$T(8,3)$& 34/9&
\begin{picture}(15,3)
\linethickness{0.25mm}
\put(0.2,0.4){\line(1,0){9}}
\put(3.3,0.4){\line(0,1){2}}
\put(3.3,0.4){\line(-1,1){1.5}}
\put(3.3,0.4){\line(1,1){1.5}}
\put(3.3,0.4){\line(-2,1){1.9}}
\multiput(0,0)(3,0){4}{$\bullet$}
\multiput(3,2)(0,0){1}{$\bullet$}
\multiput(1.4,1.7)(0,0){1}{$\bullet$}
\multiput(4.4,1.6)(0,0){1}{$\bullet$}
\multiput(1,1.1)(0,0){1}{$\bullet$}
\end{picture}
\\ \hline
9 &$T(9,4),T(9,3)$& 4&
\begin{picture}(15,3)
\linethickness{0.25mm}
\put(0.2,0.4){\line(1,0){8}}
\put(2.3,0.4){\line(0,1){2}}
\put(2.3,0.4){\line(-1,1){1.4}}
\put(2.3,0.4){\line(1,1){1.5}}
\put(2.3,0.4){\line(-2,1){1.6}}
\multiput(0,0)(2,0){5}{$\bullet$}
\multiput(2,2)(0,0){1}{$\bullet$}
\multiput(0.7,1.5)(0,0){1}{$\bullet$}
\multiput(0.3,1)(0,0){1}{$\bullet$}
\multiput(3.4,1.6)(0,0){1}{$\bullet$}
\put(10,0.4){\line(1,0){6}}
\multiput(9.5,0.1)(2,0){4}{$\bullet$}
\put(11.8,0.4){\line(0,1){2}}
\put(11.8,0.4){\line(-1,1){1.4}}
\put(11.8,0.4){\line(1,1){1.4}}
\put(11.8,0.4){\line(-2,1){1.6}}
\put(11.8,0.4){\line(2,1){1.7}}
\multiput(11.5,2)(0,0){1}{$\bullet$}
\multiput(9.8,1)(0,0){1}{$\bullet$}
\multiput(10.1,1.6)(0,0){1}{$\bullet$}
\multiput(13.3,0.9)(0,0){1}{$\bullet$}
\multiput(13.1,1.6)(0,0){1}{$\bullet$}
\end{picture}
\\ \hline
10 &$T(10,4)$& 35/8&
\begin{picture}(15,3)
\linethickness{0.25mm}
\put(0.2,0.4){\line(1,0){12}}
\put(3.3,0.4){\line(0,1){2}}
\put(3.3,0.4){\line(-1,1){1.5}}
\put(3.3,0.4){\line(1,1){1.5}}
\put(3.3,0.4){\line(-2,1){1.9}}
\put(3.3,0.4){\line(2,1){1.9}}
\multiput(0,0)(3,0){5}{$\bullet$}
\multiput(3,2)(0,0){1}{$\bullet$}
\multiput(1.4,1.7)(0,0){1}{$\bullet$}
\multiput(1,1.1)(0,0){1}{$\bullet$}
\multiput(4.4,1.6)(0,0){1}{$\bullet$}
\multiput(4.7,1)(0,0){1}{$\bullet$}
\end{picture}
\\ \hline
\end{tabular}\\
\end{center}
\caption{The extremal trees of order $n$ with maximum value of $\frac{ww_T(w)}{ww_T(v)}$.}\label{table2}
\end{table}
\vspace{0.2cm}

\begin{remark}
It should be noted that, for small values of $n$, the extremal trees proposed (and confirmed by our computational results) in Questions~\ref{th1} and \ref{th2} are exactly the same as those found with respect to the Wiener index in \cite{Barefoot1997}.
\end{remark}

We skip the similar details of our unsuccessful attempt at finding the minimum $ww_T(w)/ww_T(v)$, leaving the following question:

\begin{question}
Let $T$ be a tree of order $n\geq7$ with leaf $w$ and $v\in C_w(T)$, is the minimum value of $\frac{ww_T(w)}{ww_T(v)}$ achieved by the tree $T$ formed from a path with a pendent edge in the ``middle''?
\end{question}

We use Table~\ref{table3} to provide such extremal structures for small $n$. Note that when $n=4$ and $n=6$, $P_4$ and $T(6,4)$ minimize  $\frac{ww_T(w)}{ww_T(v)}$, respectively.

\setlength{\unitlength}{.1in}
\begin{table}
\begin{center}
\begin{tabular}{lcc}\hline
n & \mbox{Value} & \mbox{Extremal structures with order n}\\ \hline
2 & 1 &
\begin{picture}(15,2)
\linethickness{0.25mm}
\put(0.2,0.4){\line(1,0){3}}
\multiput(0,0)(3,0){2}{$\bullet$}
\end{picture}
\\ \hline
3 & 2 &
\begin{picture}(15,2)
\linethickness{0.25mm}
\put(0.2,0.4){\line(1,0){6}}
\multiput(0,0)(3,0){3}{$\bullet$}
\end{picture}
\\ \hline
4 &2&
\begin{picture}(15,3)
\linethickness{0.25mm}
\put(0.2,0.4){\line(1,0){9}}
\multiput(0,0)(3,0){4}{$\bullet$}
\end{picture}
\\ \hline
5& 13/6&
\begin{picture}(15,3)
\linethickness{0.25mm}
\put(0.2,0.4){\line(1,0){9}}
\put(3.3,0.4){\line(0,1){2}}
\multiput(0,0)(3,0){4}{$\bullet$}
\multiput(3,2)(0,0){1}{$\bullet$}
\end{picture}
\\ \hline
6 & 23/11&
\begin{picture}(15,3)
\linethickness{0.25mm}
\put(0.2,0.4){\line(1,0){12}}
\put(3.3,0.4){\line(0,1){2}}
\multiput(0,0)(3,0){5}{$\bullet$}
\multiput(3,2)(0,0){1}{$\bullet$}
\end{picture}
\\ \hline
7&29/15&
\begin{picture}(15,3)
\linethickness{0.25mm}
\put(0.2,0.4){\line(1,0){15}}
\put(6.3,0.4){\line(0,1){2}}
\multiput(0,0)(3,0){6}{$\bullet$}
\multiput(6,2)(0,0){1}{$\bullet$}
\end{picture}
\\ \hline
8& 39/21&
\begin{picture}(15,3)
\linethickness{0.25mm}
\put(0.2,0.4){\line(1,0){12}}
\put(6.3,0.4){\line(0,1){2}}
\multiput(0,0)(2,0){7}{$\bullet$}
\multiput(6,2)(0,0){1}{$\bullet$}
\end{picture}
\\ \hline
9 & 54/31&
\begin{picture}(15,3)
\linethickness{0.25mm}
\put(0.2,0.4){\line(1,0){14}}
\put(6.3,0.4){\line(0,1){2}}
\multiput(0,0)(2,0){8}{$\bullet$}
\multiput(6,2)(0,0){1}{$\bullet$}
\end{picture}
\\ \hline
10 & 69/41&
\begin{picture}(15,3)
\linethickness{0.25mm}
\put(0.2,0.4){\line(1,0){16}}
\put(8.3,0.4){\line(0,1){2}}
\multiput(0,0)(2,0){9}{$\bullet$}
\multiput(8,2)(0,0){1}{$\bullet$}
\end{picture}
\\ \hline
\end{tabular}\\
\caption{~~The extremal trees of order $n$ with minimum value of $\frac{ww_T(w)}{ww_T(v)}$.}\label{table3}
\end{center}
\end{table}\vspace{0.2cm}

\section{Concluding remark and other observations}
We have explored questions on $ww_T(.)$ similar to those studied for $w_T(.)$. It appears that such questions are generally more complicated than their analogues with respect to $w_T(.)$. This can also be seen from the following attempt to generalize Jordan's Theorem~\ref{th3}. The proof is simple but the statement is certainly not as neat as that of Jordan's.

\begin{proposition}
Let $v$ be in the hyper centroid of a tree $T$ of order $n\geq3$, $v_1, v_2, \cdots, v_l$ be the neighbors of $v$ and
 $T_1,T_2,\cdots,T_l$ be the connected components of $T-v$
with orders $n_1,n_2,\cdots,$ $n_l$, respectively.
Then for each $1\leq i\leq l$, the inequality
$$\sum\limits_{y\in V(T-T_i)}d_T(v,y)+n-n_i\geq \sum\limits_{x\in V(T_i)}d_T(v,x)$$
holds.
In addition,
$$n_i\leq\frac{1}{2}\sum\limits_{y\in V(T-T_i)}d_T(v,y)+\frac{n}{2}$$
if $\sum\limits_{x\in V(T_i)}d_T(v,x)\geq n_i$.
\end{proposition}
\begin{proof}
Note that for any $i$, $ww_T(v)\leq ww_T(v_i)$ and
\begin{align*}
ww_T(v_i)&=\sum\limits_{z\in V(T)}[d_T(v_i,z)+d^2_T(v_i,z)]\\
&=\sum\limits_{x\in V(T_i)}[d_T(v_i,x)+d^2_T(v_i,x)]+\sum\limits_{y\in V(T-T_i)}[d_T(v_i,y)+d^2_T(v_i,y)]\\
&=\sum\limits_{x\in V(T_i)}[d_T(v,x)-1+(d_T(v,x)-1)^2]+\sum\limits_{y\in V(T-T_i)}[d_T(v,y)+1+(d_T(v,y)+1)^2]\\
&=\sum\limits_{x\in V(T_i)}[d^2_T(v,x)-d_T(v,x)]+\sum\limits_{y\in V(T-T_i)}[d^2_T(v,y)+3d_T(v,y)+2]\\
&=ww_T(v)+2\sum\limits_{y\in V(T-T_i)}d_T(v,y)-2\sum\limits_{x\in V(T_i)}d_T(v,x)+2(n-n_i) .
\end{align*}
Hence $$\sum\limits_{y\in V(T-T_i)}d_T(v,y)+n-n_i\geq \sum\limits_{x\in V(T_i)}d_T(v,x) . $$
In addition, if $\sum\limits_{x\in V(T_i)}d_T(v,x)\geq n_i$, we have
$$n_i\leq\frac{1}{2}\sum\limits_{y\in V(T-T_i)}d_T(v,y)+\frac{n}{2} . $$
\end{proof}

It is also rather straightforward to deduce a recursive formula for $ww_T(.)$.
For a tree $T$ with order $n\geq3$ with root $v$ of degree $k$, let $v_i$,
 $T_i$ and $n_i$ be defined as before for $1 \leq i \leq k$. Then we have

\begin{proposition}\label{lemma7}
 Let $T$ be a  tree with order $n\geq3$, whose structure is described as above. Then
 $$ww_T(v)=2(n-1)+2\sum\limits_{i=1}^kw_{T_i}(v_i)+\sum\limits_{i=1}^kww_{T_i}(v_i) .$$
\end{proposition}
\begin{proof}
For each $1\leq i\leq k$,
\begin{align*}
\sum\limits_{x\in T_i}(d(v,x)+d^2(v,x))&=\sum\limits_{x\in T_i}[1+d(v_i,x)+(1+d(v_i,x))^2]\\
&=\sum\limits_{x\in T_i}[2+3d(v_i,x)+d^2(v_i,x)]\\
&=2n_i+2w_{T_i}(v_i)+ww_{T_i}(v_i) .
\end{align*}
Hence
\begin{align*}
ww_T(r)&=\sum\limits_{i=1}^k\sum\limits_{x\in T_i}(d(v,x)+d^2(v,x))\\
&=\sum\limits_{i=1}^k[2n_i+2w_{T_i}(v_i)+ww_{T_i}(v_i)]\\
&=2(n-1)+2\sum\limits_{i=1}^kw_{T_i}(v_i)+\sum\limits_{i=1}^kww_{T_i}(v_i) .
\end{align*}
\end{proof}


\begin{thebibliography}{1}

\bibitem{1974adam}A.~$\acute{A}$d$\acute{a}$m,  The centrality of vertices in trees. {\it Studia Sci. Math. Hung.} 9 (1974) 285--303.
\bibitem{Barefoot1997}C.~A.~Barefoot, R.~C.~Entringer, L.~A.~Sz$\acute{e}$kely, Extremal values for ratios of distances in
trees. {\it Discrete Appl. Math.} 80 (1997) 37--56.
\bibitem{Behtoei2011}A.~Behtoei, M.~Jannesari, B.~Taeri, Some New Relations between Wiener, Hyper-Wiener and Zagreb Indices. {\it MATCH Commun Math. Comput. Chem.}  65 (2011) 27--32.
\bibitem{dobrynin2001}A.~A.~Dobrynin, R.~Entringer, I.~Gutman,
Wiener index of trees: theory and applications. {\it Acta Appl.
Math.} 66 (2001) 211--249.

\bibitem{gutman2002}I.~Gutman, Relation between hyper-Wiener and Wiener index. {\it Chem. Phys. Lett.}  364 (2002) 352--356.

\bibitem{Jordan1869}C.~Jordan, Sur les assemblages de lignes.
 {\it J. Reine Angew.Math.} 70 (1869) 185--190.




\bibitem{Randic1993}M.~Randi$\acute{c}$, Novel molecular descriptor for structure-property studies. {\it Chem.Phys.Lett.} 211 (1993) 478--483.





\bibitem{szekely2013}L.~A.~Sz$\acute{e}$kely, H.~Wang, Extremal values of ratios: distance problems vs. subtree problems in trees I.  {\it Electron. J. Comb.} 20 (2013) 67, 20pp.
\bibitem{szekely2014}L.~A.~Sz$\acute{e}$kely, H. ~Wang, Extremal values of ratios: Distance problems vs. subtree
problems in trees II. {\it Discrete Mathematics} 322 (2014) 36--47.

\bibitem{wang2014}H.~Wang, The distances between internal vertices and leaves of a tree. {\it European J. Combin.} 41 (2014) 79--99.
\bibitem{wang2015}H.~Wang, Centroid,leaf-centroid,and internal-centroid. {\it Graphs and Combinatorics} 31 (2015) 783--793.
\bibitem{wiener} H.~Wiener, Structural determination of paraffin boiling point. {\it J. Amer. Chem. Soc.} 69 (1947) 17--20.

\bibitem{xu2014}K.~Xu, M.~Liu, K.~C.~Das, I.~Gutman, B.~Furtula,  A Survey on graphs extremal with respect to
distance-based topological indices. {\it MATCH Commun Math. Comput. Chem.}  71 (2014) 461--508.
\bibitem{zelinka1968}B.~Zelinka, Medians and peripherians of trees. {\it Arch.Math.} 4 (1968) 87--95.
\bibitem{zhou2004}B.~Zhou, I.~Gutman, Relation between hyper-Wiener and Wiener index. {\it Chem. Phys. Lett.}  394 (2004) 93--95.


\end{thebibliography}
\end {document}